\documentclass[reqno,10pt]{amsart}
\usepackage{amsmath,amsthm,amsfonts,color,esint,amssymb}
\usepackage[latin1]{inputenc}
\usepackage[makeroom]{cancel}
\usepackage{tikz}
\usepackage{tkz-euclide}
\usetikzlibrary{shapes.geometric,fit}

\newcommand{\ga}{\alpha}

\newcommand{\bbn}{\mathbb{N}}

\newcommand{\dhr}{\mathrel{\lhook\joinrel\relbar\kern-.8ex\joinrel\lhook\joinrel\rightarrow}}

\newcommand{\seqk}[1]{(#1 _k)_{k\in\bbn}}
\newcommand{\seqm}[1]{(#1 _m)_{m\in\bbn}}

\newtheorem{thm}{Theorem}[section]
\newtheorem{lem}[thm]{Lemma}

\newtheorem{prop}[thm]{Proposition}

\newtheorem{rem}[thm]{Remark}

\DeclareMathAlphabet{\mathpzc}{OT1}{pzc}{m}{it}

\numberwithin{equation}{section}

\begin{document}
\bibliographystyle{plain}

\title{A Class Of Elliptic Equations with Interior Degeneration}

\author{Patrick Guidotti}
\address{University of California, Irvine\\
Department of Mathematics\\
340 Rowland Hall\\
Irvine, CA 92697-3875\\ USA }
\email{gpatrick@math.uci.edu}

\begin{abstract}
A class of linear degenerate elliptic equations inspired by nonlinear
diffusions of image processing is considered. It is characterized by
an interior degeneration of the diffusion coefficient. It is shown
that no particularly natural, unique interpretation of the equation is
possible. This phenomenon is reflected in the behavior of numerical
schemes for  its resolution and points to similar issues that
potentially affect its nonlinear counterpart.
\end{abstract}


\maketitle

\section{Introduction}
The Perona-Malik equation has attracted a fair amount of interest
since its introduction in the early 1990s mainly because of an
apparent dichotomy between its mathematical ill-posedness and its
efficacy as an image processing tool. In the mathematical literature
regularizations and relaxations of various kinds have been proposed and
analyzed; we refer to \cite{G131} for an overview. Here the focus is
on the regularization introduced in \cite{GL07} which replaces the
gradient edge detection of Perona-Malik with one using fractional
derivatives. The equation reads 
\begin{equation}\label{pme}\begin{cases}
  \dot u=\nabla\cdot\bigl(\frac{1}{1+|\nabla ^{1-\varepsilon}u|^2}\nabla
  u\bigr) &\text{in }B\text{ for }t>0,\\u(0)=u_0&\text{in }B,
\end{cases}
\end{equation}
for a given, fixed $\epsilon\in(0,1]$ and an unknown curve of periodic
functions $u(t,\cdot):B\to \mathbb{R}$ on the normalized unit box (of
$\mathbb{R}^2$ in applications to image processing). The initial datum
$u_0$ is a given corrupted image that needs to be enhanced. The
classical Perona-Malik equation corresponds to setting $\epsilon =0$
in \eqref{pme} and is known for its strong edge
preservation/sharpening capabilities. This is related to its
forward-backward nature (see \cite{G131} for instance). The 
distinguishing feature of \eqref{pme} is the combination of its 
mathematical forward parabolic character, albeit degenerate, and its
strong edge preserving properties. Latter are due to the fact that
characteristic functions of smooth sets, piecewise constant functions
more in general, turn out to be stationary solutions of
\eqref{pme}. This was the motivation beyond the introduction of
\eqref{pme}. A transition between non-trivial dynamical behavior for
piecewise constant initial data has been observed to occur in
numerical experiments as the parameter $\epsilon$ crosses the
threshold value $\frac{1}{2}$. If it is smaller, such initial data are
preserved, reflecting their equilibrium status for \eqref{pme}. If it
is larger, however, (numerical) solutions typically (more on this
later) tend to exhibit fast convergence to a uniform state. In the
context of smooth solutions, this
transition from local to global well-posendess was analytically
confirmed in \cite{G122}, where global existence and convergence
to trivial steady-states are established for smooth enough initial data
for a variant of \eqref{pme} in a one-dimensional setting. This is
somewhat unsatisfactory since most interesting (numerical) solutions
of \eqref{pme} are not smooth enough and, while piecewise constant
solutions can be viewed as stationary for the evolution, no weak
solvability theory is available for any low regularity class of
functions including them. A significant impediment to the development
of a comprehensive weak solvability theory is the (conjectured)
non-existence of function spaces containing piecewise constant
functions for which weak solutions can be constructed. In this paper,
the focus is on a class of linear equations closely related to a
modification of \eqref{pme} given by 
\begin{equation}\label{mpme}\begin{cases}
  \dot u=\nabla\cdot\bigl(\frac{1}{1+N_\varepsilon^2(|\nabla
    u|)} \nabla u\bigr)=\nabla\cdot\bigl( a_\varepsilon(u)\nabla
  u\bigr) &\text{in }B\text{ for }t>0,\\u(0)=u_0&\text{in }B, 
\end{cases}
\end{equation}
where the convolution operator $N_\varepsilon$ is a Fourier
multiplication operator defined through
\begin{equation}\label{kernel}
 N_\varepsilon =\mathcal{F}^{-1}\operatorname{diag}\big[(|k|^{-\varepsilon
 })_{k\in \mathbb{Z}^n \setminus\{ 0\}}\bigr]\mathcal{F} \sim
 \frac{1}{|x|^{n-\varepsilon}}*,
\end{equation}
for $n=1,2$, where $\mathcal{F}$ denotes the Fourier transform on
$\operatorname{L}_\pi^2(B)$. As a step towards understanding 
this nonlinear equation 
for relevant non-smooth initial data, one can consider an initial
datum in the form of a characteristic function of a subset of the
circle ($n=1$) or of the torus ($n=2$), with smooth boundary $\Gamma$, and
study the linear equation
\begin{equation}\label{lmpme}\begin{cases}
  \dot u=\nabla\cdot\bigl(a_\varepsilon(u_0)\nabla u\bigr) &\text{in
  }B\text{ for }t>0,\\u(0)=u_0&\text{in }B, 
\end{cases}
\end{equation}
It will be shown that, in this case,
\begin{equation}\label{kernelasym}
 N_\varepsilon(|\nabla u|)(x)\sim
 \frac{1}{d(x,\Gamma)^{1-\varepsilon}}\text{ as }x\sim \Gamma,
\end{equation}
for the distance $d(x,\Gamma)$ to the boundary $\Gamma$ and thus that
$$
 a_\varepsilon(u_0)=\frac{1}{1+N_\varepsilon^2(|\nabla u_0|)}\sim
 d(x,\Gamma)^{2-2 \varepsilon}\text{ as }x\sim \Gamma,
$$
would hold for the corresponding diffusivity. It will be shown that,
for $\varepsilon>\frac{1}{2}$, equation \eqref{lmpme} possesses a
unique solution which instantaneously regularizes and eventually
converges to a trivial steady-state. Since certain piecewise constant
functions can also be seen as steady-states, non-uniqueness ensues.
While it seems natural to view \eqref{lmpme} as ``the'' gradient flow
engendered by the energy functional given by
$$
 \int_B \alpha |\nabla u|^2\, dx,
$$
for $\alpha=a_\varepsilon(u_0)$ and $\varepsilon>1/2$, latter does not
appear to have a preferred, unique domain of 
definition. For this reason, it cannot be claimed that \eqref{lmpme}
possesses a natural and unambiguous interpretation. It is in fact
possible to construct three distinct gradient flows compatible with the above 
energy which exhibit different behaviors. One which regularizes
initial data immediately and averages them out exponentially fast, as
is naturally expected of a heat equation, and, others, which preserves
certain discontinuities forever. This ambiguity is reflected at the
numerical level as a grid-choice phenomenon. In other words different 
solutions can be observed numerically even with the same type of
discretization depending on the choice of discretization points. 
In the ``regularizing interpretation'', the degenerate elliptic
operator can be shown to generate an analytic semigroup on
$\operatorname{L}^2_\pi$. In this case, the evolution can be viewed as
the vanishing viscosity limit for the equation with diffusivity
$\delta+a_\varepsilon(u_0)$ for $\delta>0$. While this is possibly the
most natural interpretation of the degenerate equation, others cannot
be neglected as they could help explain numerical observations. In
fact, many nonlinear diffusions have been 
utilized in image processing especially because of their ability to
preserve edges. This paper shows that, even in the linear case,
extreme care is required when using such methods as they are not
assured to deliver consistent results, nor do they provide assurances
that the output images possess properties that are naturally connected
to the underlying ``true'' image that one is purportedly trying to
recover. This will be demonstrated with a simple one dimensional
discretization.  

Elliptic and parabolic equations with interior degenerations have not
been studied extensively in the literature. The approach taken in this 
paper is most akin to that utilized by \cite{FGG12} in a one-dimensional
context in that it shows, in particular, generation of an analytic semigroup on
$\operatorname{L}^2$. Recently, a general framework for linear and nonlinear
degenerate parabolic equations has been developed in \cite{Ama16} using
different techniques based on the concept of singular manifolds. These
techniques can be adapted to construct one of the possible solutions in
a nonlinear context. This is done in \cite{GS16}.

The paper is organized as follows. In the next two sections, it is shown
how the seemingly natural energy functional for \eqref{lmpme} admits
distinct and valid interpretations which lead to different
evolutions. In the regularizing case, the associated operator will be
shown to generate analytic, contraction semigroups. Spectral
properties related to compact embeddings and the validity of a
Poincar\'e inequality will be highlighted. Additionally, two different
flows will be presented which can preserve singularities. In Section 4 the
one-dimensional case will be considered to show how numerical
implementations can indeed produce  at least two distinct types of
solutions. Interestingly, one of them is incompatible with any of the
interpretations presented in Section 3. It turns out to be compatible
with strongly degenerate equations. Section 5 deals with the vanishing
viscosity limit via $\Gamma$-convergence.
\section{The Setup and The Energy Functional}
The main focus of this paper is on linear weakly degenerate elliptic
problems with diffusivity belonging to a specific class of
functions. Let $n=1,2$ and $B=[-1,1)^n$ be the periodicity
box. Consider bounded periodic functions $\alpha:B\to[0,\infty)$ which
vanish only on a smooth closed curve $\Gamma\subset B$, if $n=2$, or on
$\Gamma=\{\pm 1/2\}$, if $n=1$, and satisfy
\begin{equation}\label{alphahypo}
 \frac{1}{c}\,d(x,\Gamma)^\sigma\leq\alpha(x)\leq
 c\,d(x,\Gamma)^\sigma,\: x\in B, 
\end{equation}
for some $1\leq c<\infty$,
$\sigma\in[0,1)$, and that are otherwise smooth on
$B\setminus\Gamma$ (at least H\"older continuous of exponent $\sigma$,
if not stated otherwise). The function $d(\cdot,\Gamma)$ given by  
$$
 d(x,\Gamma)=\inf_{y\in\Gamma}|x-y|,\: x\in B,
$$
represents the distance function to the set $\Gamma$. The collection
of all coefficient functions $\alpha$ of the above type is denoted by
$\operatorname{D}^\sigma_\pi$. For 
$\alpha\in \operatorname{D}^\sigma_\pi$ consider the elliptic problem
\begin{equation}\label{lewd}
 \begin{cases}
  -\nabla\cdot\bigl( \alpha(x)\nabla u\bigr)=f&\text{in }B,\\
  u\text{ periodic}&
 \end{cases}
\end{equation}
for $f\in \operatorname{L}^2_\pi(B)$, the space of functions which are
square integrable and periodic (hence the subscript $\pi$). The corresponding
evolutionary problem, given by
\begin{equation}\label{lpwd}
 \begin{cases}
  \dot u-\nabla\cdot\bigl( \alpha(x)\nabla u\bigr)=0&\text{in }B,\\
  u\text{ periodic},&
 \end{cases}
\end{equation}
is also of interest. In the case that $\sigma=0$, the diffusivity
cannot obviously be required to vanish on $\Gamma$ and equation
\eqref{lewd} is strongly elliptic, while, for $\sigma\in(0,1)$, it is a
so-called weakly degenerate elliptic problem. For this nomenclature
and basic results in the elliptic case, it is refered to \cite{KS87},
where a weakly degenerate  equation of type \eqref{lewd} is
characterized in particular by the conditions that
$$
 0\leq \alpha\in \operatorname{L}^1(B)\text{ and that
 }\frac{1}{\alpha}\in\operatorname{L}^1(B). 
$$
Problems \eqref{lewd} and \eqref{lpwd} are closely related to the
energy functional
\begin{equation}\label{efctnal}
 E_\alpha(u)=\int_B \alpha |\nabla u|^2\, dx.
\end{equation}
Observe that this functional can be thought of as being defined on the
weighted space 
$$
 \operatorname{H}^1_{\pi,\alpha}(B):=\big\{ u\in \operatorname{L}
 ^2_\pi(B)\, :\, |\nabla u|\in \operatorname{L}^1_\pi(B)\text{ and }\int
 \alpha(x)|\nabla u(x)|^2\, dx<\infty\big\} 
$$
which is a Banach space with respect to the norm
$$
 \| u\| _{\operatorname{H}^1_{\pi,\alpha}(B)}:=\big(\| u\| _2^2+\|
 \sqrt{\alpha}\nabla u\| _2^2\bigr )^{1/2},
$$
since the requirement that $\nabla u$ be a regular distribution does
not need to be reiterated in the norm in view of the validity of
$$
 \int _B|\nabla u|\, dx\leq \bigl( \int _B \frac{1}{\alpha(x)}\, dx
 \bigr)^{1/2} \bigl( \int _B \alpha |\nabla u|^2\, dx\bigr)^{1/2},\: u\in
 \operatorname{H}^1_{\pi,\alpha}(B),
$$
thanks to the Cauchy-Schwarz inequality.
On the other hand, if the requirement that $|\nabla u|$ be integrable
is dropped, the energy functional can be viewed as being defined on 
$$
 \widetilde{\operatorname{H}}^1_{\pi,\alpha}(B):=\big\{ u\in \operatorname{L}
 ^2_\pi(B)\, :\, \nabla u\in \mathcal{M}_\pi(B)\text{ and }\sqrt{\alpha}|\nabla u|\in
 \operatorname{L}^2_\pi(B)\big\},
$$
where $\mathcal{M}_\pi(B)$ is the space of periodic (vector-valued)
Radon measures on $B$, dual to the space $\operatorname{C}_\pi(B)^n$ of periodic
continuous functions on $B$, or, even on the larger space obtained
simply requiring that $\sqrt{\alpha}\nabla u$ be square integrable for
the distributional gradient of $u$. The main reason to consider the space
$\widetilde{\operatorname{H}}^1_\pi(B)$ along with the functional
$E_\ga$, which will be denoted by $\widetilde{E}_\alpha$ if considered
with this domain, is that it
contains the characteristic function $\chi _\Omega$ of the smooth domain
$\Omega$ bounded by the curve $\Gamma$ for $n=2$ or, of the interval
$[-1/2,1/2]$ for $n=1$. Indeed, one has that
$$
 \nabla \chi_\Omega=\nu_\Gamma \delta_\Gamma \text{ and }
 \chi_{\Omega}'=\delta_{1/2}-\delta_{-1/2},
$$
for $n=2$ and $n=1$, respectively, and therefore that
$\sqrt{\alpha}\,\nabla\chi_\Omega=0$ as well as 
$\sqrt{\alpha}\,\chi '_{[-1/2,1/2]}=0$. Here $\nu_\Gamma$
denotes the unit outward normal to $\Gamma$, while $\delta_\Gamma$
represents the line integral distribution along $\Gamma$.
Observe that these functions are non-trivial
minimizers of $\widetilde{E}_\ga$ and they might play a role in the
evolution of the corresponding gradient flow. Equation \eqref{lewd}
could arguably also be interpreted as a system for a pair $(u_i,u_o)$
of functions defined on the connected components $\Omega_i=\Omega$
and $\Omega_o$ of $B \setminus\Gamma$ and belonging to the space
$$
 \overline{\operatorname{H}}^1_{\pi,\alpha}(B):=
 \operatorname{H}^1_{\alpha}(\Omega_i)\times
 \operatorname{H}^1_{\pi,\alpha}(\Omega_o),
$$
and where the energy functional is now interpreted as
\begin{equation}\label{ebarfctnal}
  \overline{E}_\alpha(u_i,u_0)=\int _{\Omega_i}\alpha |\nabla
  u_i|^2\, dx+\int _{\Omega_0}\alpha |\nabla u_o|^2\,
  dx=E_\alpha(u_i)+E_\alpha(u_o),
\end{equation}
where the last identity holds with the understanding that the energy
functionals are for functions with the appropriate domain of
definition. This last interpretation is justified by the fact that
$\operatorname{L}^2_\pi(B)=\operatorname{L}^2(\Omega_i)\oplus
\operatorname{L}^2_\pi(\Omega _o)$, so that the energy functional, if
extended by the value $\infty$, can be thought of as being defined on
$\operatorname{L}^2_\pi(B)$.
\section{The Different Flows}
\subsection{The Regularizing Case}\label{regcase}
It is easily seen that that compactly supported test functions
belong to $\operatorname{H}^1_{\pi,\alpha}(B)$, i.e. that
$$
 \mathcal{D}(B)\subset \operatorname{H}^1_{\pi,\alpha}(B),
$$
and that
$$
 \mathcal{D}_\pi(B)=\operatorname{C}^\infty_\pi(B)\subset
 \operatorname{H}^1_{\pi,\alpha}(B), 
$$
where the subscript $\pi$ in the first space indicates that periodic test-functions
are considered. It is natural to view \eqref{lewd} with $f\equiv 0$ as
the stationarity condition for $E_\alpha$ given by \eqref{efctnal} and
defined on $\operatorname{H}^1_{\pi,\alpha}(B)$. Latter happens to be the natural
space which makes the functional coercive (see below). The form
associated to $E_\alpha$ is given by
\begin{equation}\label{form}
 a_\alpha(u,v):=\int _B\alpha\nabla u\cdot\nabla v\, dx,\: u,v\in
 \operatorname{H}^1_{\pi,\alpha}(B),
\end{equation}
and induces the operator
\begin{equation}\label{formop}
 \mathcal{A}_\alpha : \operatorname{H}^1_{\pi,\alpha}(B)\to
 \operatorname{H}^1_{\pi,\alpha}(B)'=: \operatorname{H}^{-1}_{\pi,\alpha}(B),
\end{equation}
given by
$$
 \mathcal{A}_\alpha u:=\bigl[ v\mapsto \int _B\alpha \nabla u\cdot\nabla v\,
 dx\bigr]\in \operatorname{H}^{-1}_{\pi,\alpha}(B).
$$
Clearly the form $a_\alpha$ is non-negative and symmetric. Next a few properties of
the space $\operatorname{H}^1_{\pi,\alpha}(B)$ are collected which are important for
the understanding of the weakly degenerate problem
\eqref{lewd}. Notice that proofs are mostly given for $n=2$ since the
one dimensional case is simpler and can be handled in a perfectly
analogous manner.

Take a compactly supported, radial and radially decreasing, non-negative, smooth
testfunction $\varphi\in \mathcal{D}(\mathbb{R}^n)$ with
$\operatorname{supp}(\varphi)\subset B$ and with
$$
 \int _B \varphi(y)\, dy=1.
$$
Define an associated mollifier $\varphi _m$ in the usual way  
$$
 \varphi _m(x)=m\,\varphi(mx)\text{ and } u_m(x)=\varphi _m*_\pi u(x):=\int _B\varphi
 _m(x-y)u(y)\, dy,\: x\in B.
$$
Alternatively, one can think of the convolution on the torus and write
$$
 \int _B \varphi _m(x-y|2)u(y)\, dy,\: x\in B,
$$
where $(x-y|2)$ denotes addition modulo 2 component
by component.  In order not to overburden the notation, the subscript
$\pi$ in the convolution will be dropped.
\begin{lem}\label{alphalemma}
It holds that $\alpha\,\varphi_m*
\frac{1}{\alpha}\in\operatorname{L}^\infty_\pi(B)$
\end{lem}
\begin{proof}
Consider first the one dimensional case $n=1$. Fix $\delta>0$ such
that
$$
 \frac{1}{c}|x\pm 1/2|^\sigma\leq \alpha(x)\leq c |x\pm 1/2|^\sigma, 
$$
for $c\geq 1$ and $|x\pm 1/2|\leq 3 \delta$. Now, if $|x\pm 1/2|\geq
2\delta$, one has that
$$
 \alpha(x)\varphi _m*\frac{1}{\alpha}\leq\frac{c}{2^\sigma
   \delta^\sigma}\| \alpha \| _\infty,
$$
provided $m\geq 1/\delta$. If, on the other hand $|x-1/2|<2\delta$
(the case when $|x+1/2|<2 \delta$ can be handled in the same way),
then
$$
 x\in[1/2-2/m,1/2+2/m]\cup\bigcup_{k\geq
   2}[1/2-(k+1)/m,1/2-k/m]\cup[1/2+k/m,1/2+(k+1)/m],
$$
where the union ends when the interval $[1/2-2 \delta ,1/2+2\delta]$ is
completely covered. While a finite number of intervals suffice for any finite $m$,
the number increases with $m$. If $x$  belongs to the interval
$[1/2-2/m,1/2+2/m]$, one has that
\begin{multline*}
 \alpha(x)\int_{x-1/m}^{x+1/m}\frac{\varphi_m(x-y)}{\alpha(y)}\,
 dy\leq m\int _{x-1/m}^{x+1/m}\frac{|x-1/2|^\sigma}{|y-1/2|^\sigma}\,
 dy\\\leq c\, m^{1-\sigma}\bigl[
 (1/2-x+1/m)^{1-\sigma}+(x+1/m-1/2)^{1-\sigma}\bigr]\leq c<\infty.
\end{multline*}
If, on the other hand, $x\in[1/2+k/m,1/2,+(k+1)/m]$ (or similarly, if
$x$ belongs to the interval $[1/2-(k+1)/m,1/2-k/m]$), it holds that
\begin{align*}
 \alpha(x)\int_{x-1/m}^{x+1/m}\frac{\varphi_m(x-y)}{\alpha(y)}\,
 dy&\leq m\int_{x-1/m}^{x+1/m}\frac{\alpha(x)}{\alpha(y)}\, dy\leq
 c\, m|x-1/2|^\sigma \int_{x-1/m}^{x+1/m}\frac{1}{|y-1/2|^\sigma}\,
 dy\\&\leq c \bigl( \frac{k+2}{m}\bigr) ^\sigma \bigl(
 \frac{m}{k-1}\bigr) ^\sigma
 \leq c \bigl( \frac{k+2}{k-1}\bigr) ^\sigma\leq c<\infty,\:
 k\geq 2.
\end{align*}
Since there is no restriction on $k$, the estimate is valid for
any (large) $m$ and the proof is complete for $n=1$. 

As for $n=2$, since $\Gamma$ is assumed to be a smooth, closed curve,
it possesses a tubular neighborhood $T_\Lambda(\Gamma)$ with
coordinates $(y,\lambda)$ determined by
$$
 T_\Lambda(\Gamma)=\big\{ y+\lambda\nu _\Gamma(y)\,\big |\, y\in\Gamma,\:\lambda\in
 (-\Lambda,\Lambda)\big\},
$$
where $\nu_\Lambda$ is the unit outward normal to $\Gamma$. Then, for any $x\in
T_\Lambda(\Gamma)$, it is possible to find a unique pair
$\bigl(y(x),\lambda(x)\bigr)\in\Gamma\times(-\Lambda,\Lambda)$ such that 
$$
 x=y(x)+\lambda(x)\nu_\Gamma\bigl(y(x)\bigr).
$$
It follows that any integral with respect to the two-dimensional Lebesgue measure
$dxdy$ amounts to an integral in the new coordinates with respect to the measure
$d\sigma_{\Gamma_\lambda}(y)d\lambda$, where $\sigma_{\Gamma_\lambda}$
is the line measure along 
$$
 \Gamma_\lambda=\big\{ y+\lambda\nu _\Gamma(y)\,\big |\, y\in\Gamma\big\}
$$
for $\lambda\in(-\Lambda,\Lambda)$. Notice that
$$
 d\sigma _{\Gamma _\lambda}=|\dot \gamma _\lambda(t)|dt
$$
for any parametrization $\gamma _\lambda$ of $\Gamma_\lambda$. Denote
by $\gamma$ the arc-length parametrization of $\Gamma$,
then taking $\gamma_\lambda=\gamma+\lambda\nu_\Gamma(\gamma)$ yields a
parametrization of $\Gamma_\lambda$ and
$$
 \dot \gamma_\lambda=\dot\gamma_\Gamma +\lambda \frac{d}{dt}\nu
 _\Gamma(\gamma)=\bigl[ 1+\lambda\kappa(\gamma)\bigr]
 \tau_\Gamma(\gamma), 
$$
since
$\frac{d}{dt}\nu_\Gamma(\gamma)=\kappa(\gamma)\tau_\Gamma(\gamma)$ for
the curvature $\kappa$ along $\Gamma$. It follows that
$$
 \frac{1}{c}\leq |\dot \gamma _\lambda|=|1+\lambda\kappa(\gamma)|\leq
 c,\: \lambda\in[-\Lambda,\Lambda],
$$
for some $c>1$ and $\Lambda<<1$. Consequently one has that
\begin{equation}\label{measeq}
 \frac{1}{c}\,dxdy\leq d\sigma _\Gamma d\lambda\leq c\, dxdy.
\end{equation}
With this in hand, it follows that
\begin{align*}
 \int _{\mathbb{B}(x,1/m)}\varphi_m(x-y)\frac{\alpha(x)}{\alpha(y)}\,
 dy&\sim m^2\int _{\Gamma\cap \mathbb{B}(x,1/m)}\int
 _{\lambda(x)-1/m}^{\lambda(x)+1/m}\frac{\alpha \bigl( y(x),\lambda(x)\bigr)}{\alpha(\bar
   y,\bar \lambda)}d\bar \lambda d \sigma _{\Gamma}(\bar y)\\
 &\sim m\int _{\lambda(x)-1/m}^{\lambda(x)+1/m}\frac{|\lambda(x)|^\sigma}{|\bar \lambda|
   ^\sigma}\, d\bar \lambda, 
\end{align*}
and the proof can be completed in a manner similar to that used in the
one dimensional case by considering $x$ in distance-layers around
$\Gamma$. The assumption on $\alpha$ yielding $\delta>0$ such that 
$$
 \alpha(x)\sim d(x,\Gamma)^\sigma=|\lambda(x)|^\sigma \text{ in }T_{3
   \delta}(\Gamma)
$$
was of course used in the above estimates.
\end{proof}
\begin{lem}[Density]\label{density}
The space $\mathcal{D}_\pi(B)$ of periodic test-functions is dense in
$\operatorname{H}^1_{\pi,\alpha}(B)$. 
\end{lem}
\begin{proof}
Let $u_m=\varphi_m*u$, so that $u_m\in \mathcal{D}_\pi(B)$, that
$$
 u_n\to u\,\text{ in }\operatorname{L}^2_\pi(B)\text{ as }m\to\infty.
$$
and that
$$
 \nabla u_m\to \nabla u\in \operatorname{L}^1_\pi(B)^n\text{ as
 }m\to\infty
$$
for any $u\in \operatorname{H}^1_{\pi,\alpha}(B)$. Without loss of generality,
it can be assumed that $\nabla u_m\to\nabla u$ pointwise almost
everywhere (otherwise just take the appropriate subsequence). Then
$$
 \alpha\,|\partial _j u_m|^2\longrightarrow\alpha\,|\partial _j u|^2\text{
   a.e. for }j=1,2\text{ as }m\to\infty.
$$
If it were possible to show that
\begin{equation}\label{esti}
 \alpha\, |\partial_ju_m|^2\leq g_m,\: m\in \mathbb{N},
\end{equation}
for nonnegative measurable functions $g_m$ which converge pointwise almost
everywhere to $g\in \operatorname{L}^1_\pi(B)$ and for which
$$
 \int_Bg_m\, dx \longrightarrow \int _Bg\, dx\text{ as }m\to\infty,
$$
then the generalized Dominated Convergence Theorem would imply that
$$
 \int _B \alpha\, |\partial_ju_m|^2\, dx \longrightarrow \int _B \alpha\,
 |\partial_ju|^2\, dx\text{ as }m\to\infty,
$$
which, together with the almost everywhere convergence, would yield
$$
 \sqrt{\alpha}\,\nabla u_m\to\sqrt{\alpha}\,\nabla u\text{ in
 }\operatorname{L}^2_\pi(B)^n\text{ as }m\to\infty, 
$$
and the claim. Going back to \eqref{esti}, Lemma \ref{alphalemma} gives
\begin{multline*}
 \alpha(x)\big |\int_B\varphi_m(x-y)\partial_ju(y)\, dy\big |^2\leq\int _B
 \varphi _m(x-y)\frac{\alpha(x)}{\alpha(y)}\, dy\int _B
 \varphi_m(x-y)\alpha(y) |\partial_j u(y)|^2\, dy\\
\leq c\,\varphi_m*\bigl( \alpha\, |\partial_j u|^2\bigr),\: j=1,2,\:
m\in \mathbb{N},
\end{multline*}
and $u\in \operatorname{H}^1_{\pi,\alpha}(B)$ ensures that
$$
 \varphi_m*\bigl(\alpha\, |\partial_ju|^2 \bigr)\longrightarrow \alpha\,
 |\partial_ju|^2\text{ in }\operatorname{L}^1_\pi(B)\text{ as }m\to\infty,
$$
as desired.
\end{proof}
\begin{lem}[Compact Embedding]\label{compactness}
The embedding $\operatorname{H}^1_{\pi,\alpha}(B)\hookrightarrow
\operatorname{L}^2_\pi(B)$ is compact. 
\end{lem}
\begin{proof}
In view of assumption \eqref{alphahypo} on the weight function
$\alpha$, an exponent $p>1$ can be found such that
$$
 \int _B \frac{1}{\alpha(x) ^p}\, dx<\infty.
$$
Then one has that $|\nabla u|\in \operatorname{L}^{1+\delta}_\pi(B)$ for some $\delta>0$
small enough since
\begin{multline*}
 \int _B |\nabla u(x)|^{1+\delta}\, dx\leq \int _B
 \bigl(\frac{\sqrt{\alpha(x)}}{\sqrt{\alpha(x)}}\bigr)^{1+\delta}|\nabla
  u(x)|^{1+\delta}\, dx\\\leq\bigl( \int _B \alpha(x)^{-\frac{1+\delta}{1-\delta}}\, dx
  \bigr)^{\frac{1-\delta}{2}}\bigl(\int _B \alpha(x)|\nabla u(x)|^2\, dx\bigr)
  ^{\frac{1+\delta}{2}}<\infty,
\end{multline*}
provided $\frac{1+\delta}{1-\delta}<p$, which is always possible for a small enough
$\delta$. This shows that $u\in \operatorname{W}^{1,1+\delta}_\pi(B)$
and the claim therefore follows from the compactness part of Sobolev
embedding theorem observing 
that $2<(1+\delta)^*=\frac{n(1+\delta)}{n-1-\delta}$ is valid as long as $n<
2\frac{1+\delta}{1-\delta}$. This is always the case for dimensions $n=1,2$.
\end{proof}
\begin{lem}[Existence of Traces]\label{trace}
Any function $u\in \operatorname{H}^1_{\pi,\alpha}(B)$ admits a trace $\gamma
_\Gamma(u)\in \operatorname{L}^2(\Gamma)$ on the degeneration set $\Gamma$.
\end{lem}
\begin{proof}
Using the coordinates introduced in the proof of Lemma
\ref{alphalemma} for the tubular neighborhood $T_\Lambda(\Gamma)$ of
$\Gamma$, take $u\in \mathcal{D}_\pi(B)$ and
let $\seqm{\lambda}$ be a sequence in $(-\Lambda,\Lambda)\setminus\{
0\}$ such that $\lambda_m\to 0$ as $m\to\infty$. Then
$$
 u(y,\lambda _k)-u(y,\lambda _l)=\int _{\lambda _l}^{\lambda _k}\partial _\lambda
 u(y,\lambda)\, d\lambda.
$$
It follows that
\begin{multline*}
 \| u(\cdot,\lambda _k)-u(\cdot,\lambda _l)\|
 _{\operatorname{L}^2(\Gamma)}\leq\int 
 _\Gamma\big |\int _{\lambda _l}^{\lambda _k}\partial _\lambda
 u(y,\lambda)\, d\lambda\big |^2\, d\sigma _\Gamma(y)\\\leq\int
 _\Gamma\Bigl(\int _{\lambda _l}^{\lambda _k} 
 \frac{1}{\alpha(y,\lambda)}\, d\lambda\Bigr)\Bigl(\int _{\lambda
   _l}^{\lambda _k}\alpha(y,\lambda)| \partial _\lambda
 u(y,\lambda)|^2\, d\lambda\Bigr)\, d\sigma _\Gamma(y).
\end{multline*}
Noticing that 
$$
 d(x,\Gamma)=|\lambda(x)|\text{ for }x\in T_\Gamma(\Lambda),
$$
assumption \eqref{alphahypo} on the diffusivity $\alpha$ now implies that
$$
 \frac{1}{\alpha (y,\lambda)}\leq c\, \frac{1}{|\lambda| ^\sigma},\:
 (y,\lambda)\in\Gamma\times(-\Lambda,\Lambda),
$$
for a constant $c$ independent of $(y,\lambda)$ and thus
$$
 \int _{\lambda _l}^{\lambda _k} \frac{1}{\alpha (y,\lambda)}\,
 d\lambda\to 0\text{ as }k,l\to\infty.
$$
As for the remaining integral one has
$$
 \int _{\lambda _l}^{\lambda _k} \int _\Gamma \alpha(y,\lambda)
 | \partial _\lambda u(y,\lambda)|^2 \, d\sigma _\Gamma(y)
 d\lambda\leq c\,\int _{T_\Lambda} \alpha |\nabla u|^2\, dx dy\leq c\,\|
 u\| ^2_{\operatorname{H}^1_{\pi,\alpha}(B)},
$$
using \eqref{measeq} and that $\partial _\lambda u(y,\lambda)=\nabla
u(y,\lambda)\cdot\nu _\Gamma(y)\leq|\nabla u(y,\lambda)|$.
This shows that $\bigl(u(\cdot,\lambda _m)\bigr)_{m\in \mathbb{N}}$ is
a Cauchy sequence in $\operatorname{L}^2(\Gamma)$ and thus that there
exists a limit, which we denote by $\gamma_\Gamma(u)\in
\operatorname{L}^2(\Gamma)$, such that 
$$
 \gamma _{\Gamma _{\lambda _k}}(u)\to\gamma _\Gamma(u)\text{ as
 }k\to\infty. 
$$
Observe that the trace operators $\gamma _{\Gamma _{\lambda _k}}$ are
well-defined for any $k\in \mathbb{N}$ since $u\in
\operatorname{H}^1\bigl(B\setminus T_\varepsilon(\Gamma)\bigr)$ for any
$\varepsilon>0$ and $\Gamma_{\lambda_k}\subset B \setminus
T_\varepsilon(\Gamma)$ for $\varepsilon<<1$.
The construction of a trace for $u$ is therefore
completed in the smooth case. The rest follows by the density
established in Lemma \ref{density} 
\end{proof}
\begin{lem}[Poincar\'e Inequality]\label{poincare}
It holds that
$$
 \| u\| _{\operatorname{L}^2_\pi(B)}\leq c\, \|\alpha\nabla u\|
 _{\operatorname{L}^2_\pi(B)},\: u\in \operatorname{H}^1_{\pi,\alpha,0}(B) 
$$
where
$$
 \operatorname{H}^1_{\pi,\alpha,0}(B)=\big\{ u\in
 \operatorname{H}^1_{\pi,\alpha}(B)\, \big |\, \int _B u(x)\, dx=0\big\}. 
$$
\end{lem}
\begin{proof}
Towards a contradiction assume that the inequality does not hold, that
is, that there is a sequence $\seqk{u}$ in
$\operatorname{H}^1_{\pi,\alpha,0}(B)$ such that
$$
 \| u_k\| _{\operatorname{L}^2_\pi(B)}\geq k\|\sqrt{\alpha}\,\nabla u_k\|
 _{\operatorname{L}^2_\pi(B)}.
$$
Define $v_k=u_k/\| u_k\|_{\operatorname{L}^2_\pi(B)}$ so that
$$
 \| v_k\| _{\operatorname{L}^2_\pi(B)}=1\text{ and } \|\sqrt{\alpha}\,\nabla v_k\|
 _{\operatorname{L}^2_\pi(B)} =\frac{\|\sqrt{\alpha}\,\nabla u_k\|
   _{\operatorname{L}^2_\pi(B)}}{\| u_k\| _{\operatorname{L}^2_\pi(B)}} \leq 
 \frac{1}{k},\: k\in \mathbb{N}.
$$
In particular it holds that $\| v_k\|_{\operatorname{H}^1_{\pi,\alpha}(B)}\leq
c<\infty$ for $k\in \mathbb{N}$ and, by the weak sequential
compactness of bounded sets in Hilbert spaces, there
must be $v_\infty\in \operatorname{H}^1_{\pi,\alpha,0}(B)$ such that
$$
 v_k \rightharpoonup  v_\infty\text{ in
 }\operatorname{H}^1_{\pi,\alpha,0}(B)\text{ 
   along a subsequence}.
$$
The convergence of the norms then yields that  $v_k\to v_\infty$ in
$\operatorname{H}^1_{\pi,\alpha,0}(B)$ along the subsequence. In this
case $\nabla v_\infty=0$ almost everywhere in the two connected
components $\Omega_i$ and $\Omega_o$ of $B\setminus\Gamma$ since, by
weak lower semicontinuity, it holds that $\|\alpha\nabla v_\infty\|
_{\operatorname{L}^2_\pi(B)}\leq \liminf_{k\to\infty}\| \alpha\nabla
v_k\|_{\operatorname{L}^2_\pi(B)}$ along the subsequence. Thus 
$$
 v_\infty(x)=\begin{cases} c_i,&x\in \Omega_i\\ c_o,&x\in
   \Omega _o\end{cases} 
$$
and it can be inferred from Lemma \ref{trace} that necessarily
$c_i=c_o$ since otherwise $v_\infty$ would not possibly possess a
well-defined trace on $\Gamma$. The mean zero 
condition finally yields that the constant must be $0$. This clearly
contradicts the fact that $\| v_\infty\|
_{\operatorname{L}^2_\pi(B)}=1$ and concludes the proof. 
\end{proof}
The above lemma clearly implies that
$\|\sqrt{\alpha}\,\nabla\cdot\|_{_{\operatorname{L}^2_\pi(B)}}$ is an
equivalent norm on $\operatorname{H}^1_{\pi,\alpha}(B)$.
The Poincar\'e inequality implies that the restriction of the
nonnegative, continuous, and symmetric bilinear form \eqref{form} to
$\operatorname{H}^1_{\pi,\alpha,0}(B)\times
\operatorname{H}^1_{\pi,\alpha,0}(B)$ is elliptic and therefore
induces a self-adjoint linear operator
$$
 \mathcal{A}_{\alpha,0}:\operatorname{H}^1_{\pi,\alpha,0}(B)\to
 \operatorname{H}^1_{\pi,\alpha,0}(B)'=:\operatorname{H}^{-1}_{\pi,\alpha,0}(B),\:
 u\mapsto a_\alpha(u,\cdot), 
$$
such that
$$
 \mathcal{A}_{\alpha,0}:\operatorname{H}^1_{\pi,\alpha,0}(B)\to
 \operatorname{H}^{-1}_{\pi,\alpha,0}(B)
$$
is invertible and has, by Lemma \ref{compactness}, compact
resolvent. Here it holds that 
$$
 \operatorname{H}^{-1}_{\pi,\alpha,0}(B)=\big\{ u\in
 \operatorname{H}^1_{\pi,\alpha}(B)'\, \big |\, \langle
 u,\mathbf{1}\rangle=0\big\} 
$$
where $\mathbf{1}$ denotes the constant function with value $1$. It
follows that
$$
 \mathcal{A}_\alpha=\sum _{k=1}^\infty\mu _k(\cdot |\varphi _k)\varphi _k,
$$
for $(\mu _k,\varphi _k)$ eigenvalue/eigenvector pairs of $A_\alpha$ with
$$
 0<\mu _1\leq\mu_2\leq\dots\mu _k\to\infty\:(k\to\infty),
$$
and where $\frac{1}{\sqrt{2^n}}\mathbf{1}=:\varphi_0,\varphi
_1,\varphi _2,\dots$ is an orthonormal basis for
$\operatorname{H}^{-1}_{\pi,\alpha}(B)$. The
$\operatorname{L}^2_\pi(B)$-realization $A_\alpha$ of
$\mathcal{A}_\alpha$ will be particularly useful and is defined by
$A_\alpha u=\mathcal{A}_\alpha u$ for 
\begin{align}\label{domain}
 u\in \operatorname{dom}(A_\alpha)&=\Big\{ u\in
 \operatorname{H}^{1}_{\pi,\alpha}(B)\, \big |\,
 a_\alpha(u,\cdot)\text{ is }\operatorname{L}^2_\pi(B)
 \text{-continuous}\Big\} \\
 &=\Big\{ u\in \operatorname{H}^{1}_{\pi,\alpha}(B)\, \big |\,
 \operatorname{div}(\alpha\nabla u)\in
 \operatorname{L}^2_\pi(B)\Big\}=:\operatorname{H}^2_{\pi,\alpha}(B).
\end{align}
The second equality requires a proof. Assume that
$\operatorname{div}(\alpha \nabla u)\in\operatorname{L}^2_\pi(B)$, then
$$
 \int _B \alpha\underset{\in
   \operatorname{L}^2_{\pi,\alpha}(B)}{\underbrace{\nabla 
     u}}\cdot\underset{\in
   \operatorname{L}^2_{\pi,\alpha}(B)}{\underbrace{\nabla  
     v}}\,dx =-\int _B \underset{\in
   \operatorname{L}^2_\pi(B)}{\underbrace{\operatorname{div}(\alpha\nabla
     u)}}\underset{\in \operatorname{L}^2_\pi(B)}{\underbrace{v}}\, dx,\: v\in
 \mathcal{D}_\pi(B),
$$
and thus
$$
 |a_\alpha(u,v)|\leq\| \operatorname{div}(\alpha\nabla 
     u) \| _{\operatorname{L}^2_\pi(B)}\| v\|
 _{\operatorname{L}^2_\pi(B)},\: v\in \mathcal{D}_\pi(B).
$$
Conversely, if
$$
 \big |\int _B\alpha\nabla u\cdot\nabla v\, dx\big |\leq c\,\| v\|
 _{\operatorname{L}^2_\pi(B)},\: v\in \mathcal{D}_\pi(B),
$$
then there is $w\in \operatorname{L}^2_\pi(B)$ such that
$$
 \int _B\alpha\nabla u\cdot\nabla v\, dx=\int _B wv\, dx,\: v\in
 \mathcal{D}_\pi(B), 
$$
which entails that $\operatorname{div}(\alpha\nabla u)=-w\in
\operatorname{L}^2_\pi(B)$. Clearly
$A_\alpha:\operatorname{dom}(A_\alpha)\subset
\operatorname{L}^2_\pi(B)\to \operatorname{L}^2_\pi(B)$ is given by 
\begin{equation}\label{specrep}
 A_\alpha u=\sum _{k=1}^\infty\mu_k\,\underset{:=\hat
   u_k}{\underbrace{(u|\varphi _k)}}\,\varphi_k,\:
 u\in\operatorname{dom}(A_\alpha),
\end{equation}
and thus
$$
 e^{-tA_\alpha}u=\hat u_0+\sum _{k=1}^\infty e^{-\mu_kt}\hat
 u_k\varphi_k,\: u\in\operatorname{L}^2_\pi(B). 
$$
Notice that
\begin{gather*}
 \| u\| _{\operatorname{L}^2_\pi(B)}=\| \seqk{\hat u}\| _{l_2(\mathbb{N})}\text{ and
 }\\\| e^{-tA_\alpha}u\| _{\operatorname{L}^2_\pi(B)}=\| \bigl(e^{-\mu_kt}\hat
   u_k\bigr) _{k\in \mathbb{N}}\| _{l_2(\mathbb{N})}\leq \| \seqk{\hat
     u}\| _{l_2(\mathbb{N})}=\| u\| _{\operatorname{L}^2_\pi(B)}.
\end{gather*}
Thus $\big\{ T_1(t):=e^{-tA_\alpha}\, |\, t\geq 0\big\}$ is a contraction semigroup
and, since,
$$
 \| tA_\alpha e^{-tA_\alpha}u\| _{\operatorname{L}^2_\pi(B)}=\| \bigl(t\mu
 _ke^{-\mu_kt}\hat u_k\bigr) _{k\in \mathbb{N}}\|
 _{l_2(\mathbb{N})}\leq c\, \| u\| 
 _{\operatorname{L}^2_\pi(B)},\: t>0,
$$
it is also analytic (see \cite{Gol85}). Strong continuity can also be
easily derived via the spectral representation
\eqref{specrep}. Summarizing 
\begin{thm}\label{linex}
The operators $A_\alpha$ and $\mathcal{A}_\alpha$ generate strongly continuous
analytic contraction semigroups on $\operatorname{L}^2_\pi(B)$ and on
$\operatorname{H}^{-1}_{\pi,\alpha}(B)$, respectively. In particular, for any given $u_0\in
\operatorname{L}^2_\pi(B)\, \bigl[ \operatorname{H}^{-1}_{\pi,\alpha}(B)\bigr]$,
there is a unique solution $u\in \operatorname{C}\bigl(
[0,\infty),\operatorname{L}^2_\pi(B)\bigr)\:\bigl[ \operatorname{C}\bigl(
[0,\infty),\operatorname{H}^{-1}_{\pi,\alpha}(B)\bigr)\bigr]$ of the abstract Cauchy
problem  
$$
 \begin{cases} \dot u=A_\alpha u\text{ in } \operatorname{L}^2_\pi(B)\,\bigl[ \dot
   u=\mathcal{A}_\alpha u\text{ in }\operatorname{H}^{-1}_{\pi,\alpha}\bigr],&t>0,\\ 
 u(0)=u_0,\end{cases}
$$
satisfying
$$
 u\in\operatorname{C}^1\bigl( (0,\infty),\operatorname{L}^2_\pi(B)\bigr)\cap
 \operatorname{C}\bigl( (0,\infty),\operatorname{H}^2_{\pi,\alpha}(B)\bigr)\,\Bigl[ 
 \operatorname{C}^1\bigl((0,\infty),\operatorname{H}^{-1}_{\pi,\alpha}(B)\bigr)
 \cap\operatorname{C}\bigl((0,\infty),\operatorname{H}^{1}_{\pi,\alpha}(B)\bigr)
 \Bigr].
$$
Moreover, one always has that
$$
 u(t,u_0)\longrightarrow\frac{1}{2^n}\langle u_0,\mathbf{1}\rangle
 \text{ as }t\to\infty, 
$$
in $\operatorname{L}^2_\pi(B)$ $\bigl[\operatorname{H}^{-1}_{\pi,\alpha}(B)\bigr]$. 
\end{thm}
\begin{rem}
Depending on the functional setting chosen, the above theorem yields a strong or weak
solution of the initial boundary value problem
\begin{equation}\label{lwdpe}
\begin{cases}
 \dot u=\nabla\cdot\bigl(\alpha(x) \nabla u\bigr)&\text{ in }B\text{ for }t>0,\\
 u(0,\cdot)=u_0&\text{ in }B,
\end{cases}
\end{equation}
respectively.
\end{rem}
\begin{rem}
Notice how a piecewise constant initial datum is instanteneously
regularized in spite of the fact that it is a steady-state of the
equation. While the theorem ensures well-posedness in the specified
classes of functions, the existence of additional solutions is
observed also in numerical discretizations of the equation. More on
this in Section 4.
\end{rem}
\begin{rem}
While the semigroup is analytic, it does not follow that solutions are
$\operatorname{C}^\infty$. This is due to the fact that the
eigenfunctions are not smooth where the coefficient $\alpha$
vanishes. 
\end{rem}
\subsection{The Singular Case}
It was already noticed that $\widetilde{E}_\alpha$ has additional
minimizers as compared to $E_\alpha$, for which only constant
functions are minimizing. Let
$$
 {\bf h}(x):=
 \begin{cases}
   \frac{1}{\sqrt{2^n}}\frac{|B \setminus \Omega|^{1/2}}{|\Omega|^{1/2}},&
   x\in \Omega_i,\\
   -\frac{1}{\sqrt{2^n}}\frac{|\Omega|^{1/2}}{|B \setminus
     \Omega|^{1/2}},&x\in \Omega_o,
 \end{cases}
$$
with the understanding that $\Omega_i=\Omega$ for $n=2$,
$\Omega_i=[-1/2,1/2]$ for $n=1$, that $\Omega_o=B^n \setminus
\Omega$ for $n=1,2$, and that $|S|$ is the Lebesgue measure of the
measurable set $S$. Then ${\bf h}$ is a minimizer of $\widetilde{E}_\alpha$
and satisfies
$$
 \int _B {\bf h} (x)\, dx=0\text{ and }\int _B {\bf h}^2(x)\, dx=1.
$$
It is then possible to consider the modified energy functional
$$
 \widetilde{\mathcal{E}}_\alpha(u,c):=\int _B \alpha \, |\nabla u|^2\,
 dx,\: u\in \operatorname{H}^1_{\pi,\alpha}(B),\: c\in \mathbb{R},
$$
on the space $\operatorname{H}^1_{\pi,\alpha}(B)\oplus\mathbb{R}
{\bf h}\subset \operatorname{L}^2_\pi(B)$ and the
associated gradient flow
$$\begin{cases}
 \dot u=\nabla\cdot \bigl( \alpha \nabla u\bigr)&\\
 \dot c =0.&
 \end{cases}
$$
In this case, the solution $u$ to an initial value $u_0+c {\bf h}\in
\operatorname{H}^1_{\pi,\alpha}(B)\oplus\mathbb{R}{\bf h}$, 
would satisfy
$$
 u(t,\cdot)\longrightarrow \frac{1}{2^n}\int _B u_0(x)\, dx+c {\bf h},
$$
thus preserving the singular component during the entire evolution. While
this is a perfectly acceptable interpretation of equation
\eqref{lpwd}, it has some serious shortcomings. Most notably, the
natural semigroup associated to it and given by
$$T_2(t)=
\begin{bmatrix}
  e^{-t A_\alpha}&0\\0&1
\end{bmatrix}\text{ on }\operatorname{H}^1_{\pi,\alpha}(B)\oplus
\mathbb{R}\, {\bf h}
$$
cannot be reasonably extended to $\operatorname{L}^2_\pi(B)$ as it is
not $\operatorname{L}^2_\pi$-continuous as follows from
\begin{multline*}
 \| T_2(t)[{\bf h}_m-{\bf h}]\|_2=\| e^{-t A_\alpha}{\bf h}_m-{\bf h}\|_2=\|
 \sum_{k=1}^me^{-t \lambda _k}\hat {\bf h}_k\varphi_k-\sum _{k=1}^\infty\hat
 {\bf h}_k\varphi_k\|_2 \\\longrightarrow \bigl[  \sum_{k=1}^\infty
 (1-e^{-t\lambda_k})^{2}\hat {\bf h}_k^2 \bigr] ^{1/2}\neq 0\text{ for any
 }t>0\text{ as }m\to\infty,
\end{multline*}
where
$$
 \operatorname{H}^1_{\pi,\alpha}(B)\ni {\bf h}_m:=\sum _{k=1}^m\hat
 {\bf h}_k\varphi_k \longrightarrow {\bf h}\text{ in
 }\operatorname{L}^2_\pi(B)\text{ as }m\to\infty.
$$
\subsection{The Split Case}
Given the diffusion coefficient $\alpha\in
\operatorname{D}^\sigma_\pi$, one can consider the energy functional 
$$
 \overline{E}_\alpha(u^i,u^o)=\frac{1}{2}\int _{\Omega^i}\alpha |\nabla u^i|^2\,
 dx+\frac{1}{2}\int _{\Omega^o}\alpha |\nabla u^o|^2\, dx,\: (u^i,u^o)\in
 \operatorname{H}^1_\alpha(\Omega^i)\times
 \operatorname{H}^1_{\alpha,\pi}(\Omega^o),
$$
where $\Omega^i$ and $\Omega^o$ have previously been
defined. Arguments perfectly analogous to
those used in Section \ref{regcase} can be used to prove the following result.
\begin{thm}
The restriction of the functional $\overline{E}_\alpha$ to
$\operatorname{H}^1_{\alpha,0}(\Omega^i)\times
\operatorname{H}^1_{\alpha,\pi,0}(\Omega^o)\to \mathbb{R}$ is coercive
and the operator induced by $\overline{E}_\alpha$
$$
 \mathcal{A}_\alpha=\operatorname{diag}(\mathcal{A}_\alpha
 ^i,\mathcal{A}_\alpha ^o):
 \operatorname{H}^1_{\alpha}(\Omega^i)\times
 \operatorname{H}^1_{\alpha,\pi}(\Omega^o)\to
 \operatorname{H}^{-1}_{\alpha}(\Omega^i)\times
 \operatorname{H}^{-1}_{\alpha,\pi}(\Omega^o)
$$
and
$$
 A_\alpha=\operatorname{diag}(A_\alpha ^i,A_\alpha ^o):
 \operatorname{dom}(A_\alpha ^i)\times \operatorname{dom}(A_\alpha
 ^o)\to\operatorname{L}^2(\Omega^i)\times
 \operatorname{L}^2_\pi(\Omega^o) \hat{=}\operatorname{L}^2_\pi(B)
$$
with
$$
 \operatorname{dom}(A_\alpha ^l)=\big\{ u\in
 \operatorname{L}^2(\Omega^l)\, \big |\,\operatorname{div}\bigl( \alpha
 \nabla u \bigr) \in \operatorname{L}^2(\Omega^l)\big\},\: l=i,o.
$$
generate analytic contraction semigroups on
$\operatorname{H}^{-1}_{\alpha}(\Omega^i)\times 
 \operatorname{H}^{-1}_{\alpha,\pi}(\Omega^o)$ and on
 $\operatorname{L}^2_\pi(B)$, repectively. Call the latter
 $T_3(t)$. It follows that the system
\begin{equation}\label{lsmpme}
  \begin{cases}
    u^i_t=\nabla\cdot \bigl( \alpha \nabla u^i \bigr)&\text{ in
    }\Omega^i\text{ for }t>0,\\
    u^o_t=\nabla\cdot \bigl( \alpha \nabla u^o \bigr)&\text{ in
    }\Omega^o\text{ for }t>0,\\
   \lim _{x\to  \Gamma }\alpha(x) \partial _{\nu _\Gamma }u^i(x)=0&\\
   \lim _{x\to  \Gamma }\alpha(x) \partial _{\nu _\Gamma }u^o(x)=0&\\
    u^i(0,\cdot)=u^i_0&\text{ in }\Omega^i\\
    u^o(0,\cdot)=u^o_0&\text{ in }\Omega^o
  \end{cases}
\end{equation}
is uniquely (weakly) solvable for any $u_0\in
\operatorname{L}_\pi^2(B)$ (or, more in general, for an initial datum $u_0\in 
\operatorname{H}^{-1}_{\alpha}(\Omega^i)\times
 \operatorname{H}^{-1}_{\alpha,\pi}(\Omega^o)$), and the solution
 converges to a trivial steady-state in each subdomain, that is,
$$
 T_3(t)u_0 \longrightarrow \bigl(\frac{1}{|\Omega_i|}\int
 _{\Omega_i}u_0(x)\, dx \bigr)
 \chi_{\Omega_i}+\bigl(\frac{1}{|\Omega_o|}\int 
 _{\Omega_o}u_0(x)\, dx \bigr) \chi_{\Omega_o},
$$
for $u_0\in \operatorname{L}^2_\pi(\Omega)$.
\end{thm}
In this interpretation, one obtains an evolution on
$\operatorname{L}^2_\pi(B)$ for which an initial datum that is constant on each of
the domains $\Omega^l$, $l=i,o$, is a stationary solution and won't be
regularized nor evolved.
\begin{rem}
Taking the system point of view, it is possible to recover the
interpretation of Section 1 by defining the energy functional
$\overline{E}_\alpha$ on
$$
 \big\{ u=(u^i,u^o)\,\big |\, u\in \operatorname{H}^1_\alpha(\Omega^i)\times
 \operatorname{H}^1_{\alpha,\pi}(\Omega^o)\text{ and }\gamma _\Gamma
 (u^i)= \gamma_\Gamma (u^o)\big\}.
$$
This means that ``continuity'' across the interface has to be
explicitly enforced.
\end{rem}
\begin{rem}
Notice that the behavior of solutions in this interpretation
is possibily what one would like to see from an application to image
processing point of view in that solutions not only tend to become
piecewise constant but the constants are also the local averages of the
initial datum in the corresponding regions of constancy.
\end{rem}
\section{A Numerical Remark}
The non-uniqueness phenomenon highlighted above will be investigated
for a spatial semi-discretization of \eqref{lewd} in a one-dimensional
setting. The observation extends to the two-dimensional setting with
the appropriate modifications. It is observed from the outset that,
even the same numerical scheme, can produce two distinct
solutions. One is the discrete counterpart of the regularizing
solution of Section \ref{regcase}; the other ``feels'' the presence of
the singular solution $h$ but, interestingly, is not the compatible
with any of the three interpretations of equation \eqref{lpwd} given
above. An explanation of its origin will follow in the later part of
this section. 

Letting $n=1$ and $\alpha\in D^\sigma _\pi$ as in the previous
sections and choosing 
$$
 {\bf h}=\frac{1}{\sqrt{2}}\chi_{\Omega_i}-\frac{1}{\sqrt{2}}\chi_{\Omega_o},
$$
Theorem \ref{linex} yields a solution
$$
 u\in \operatorname{C}\bigl([0,\infty),\operatorname{L}^2_\pi(-1,1)
 \bigr)\cap\operatorname{C}^1\bigl((0,\infty),\operatorname{L}^2_\pi
 (-1,1)\bigr)\cap\operatorname{C}\bigl((0,\infty),
 \operatorname{H}^2_{\pi,\alpha}(-1,1)\bigr),
$$
for
\begin{equation}\label{1dl}
 \begin{cases}
 \dot u=\partial _x\bigl( \alpha(x)\partial _x u\bigr)&\text{ in
 }(-1,1)\text{ for }t>0,\\ u\text{ periodic},&\\
  u(0,\cdot)={\bf h},&\end{cases}
\end{equation}
where 
$$
 \operatorname{H}^2_{\pi,\alpha}(-1,1)=\big\{
 u\in\operatorname{H}^1_{\pi,\alpha}(-1,1)\, \big |\, \alpha u'\in 
 \operatorname{H}^1_\pi(-1,1)\big\},
$$
as follows from characterization \eqref{domain} of the previous
section. Theorem \ref{linex} then implies that
$$
 u(t,\cdot)=T_1(t) {\bf h}\to \frac{1}{2}\int _{-1}^1 {\bf h} (x)\,
 dx=0\text{ as }t\to\infty.
$$
It, however, also holds that
$$
 \alpha {\bf h} '=\alpha(\delta _{-1/2}-\delta _{1/2})=\alpha(-1/2) \delta
 _{-1/2}-\alpha(1/2) \delta _{1/2}=0, 
$$
so that $u(t,\cdot)\equiv u_0$ is a stationary solution of \eqref{1dl}. This
non-uniqueness is reflected at the numerical level. Indeed set
\begin{align*}
 x^m_i&=\frac{i}{m},\: i=-m,-m+1,\dots,m-1,m,\\d_m&=1/m,\\\alpha
  ^m_i&=\alpha(x^m_i). 
\end{align*}
Then
\begin{equation}\label{1ddl}
 u^m_t=\Delta ^{m,-}\bigl(\alpha ^m\Delta ^{m,+}(u^m)\bigr)
\end{equation}
is the gradient flow to the discrete energy functional given by
\begin{equation}\label{1dden}
 E^m_\alpha(u^m)=\frac{1}{2}\sum _{i=-m}^{m-1}[\alpha^m_i\Delta _i^{m,+}(u^m)]^2d_m
\end{equation}
where
\begin{align*}
 \Delta ^{m,+}_i(u^m)&=\frac{u^m_{i+1}-u^m_i}{d_m},\: i=-m,\dots,m-1,\\
  \Delta ^{m,-}_i(u^m)&=\frac{u^m_{i}-u^m_{i-1}}{d_m},\: i=-m,\dots,m-1,
\end{align*}
with the understanding that
$$
 u^m_{-m-1}=u^m_{m-1}\text{ and that }u^m_{m+1}=u^m_{-m+1},
$$
enforcing periodicity. The ordinary differential equation \eqref{1ddl}
is a spatial semi-discretization of \eqref{1dl}, and \eqref{1dden} is
one of the continuous energy functional \eqref{efctnal} on
$\operatorname{H}^1_{\pi,\alpha}(-1,1)$. This is seen by computing
\begin{multline*}
 \left.\frac{d}{d\epsilon}\right|_{\epsilon=0}E^m_\alpha(u^m+\epsilon\varphi
 ^m)=\sum _{i=-m}^{m-1}\alpha
 ^m_i\Delta_i^{m,+}(\varphi^m)\Delta_i^{m,+}(u^m)d_m\\=-\sum
 _{i=-m}^{m-1}\bigl[\alpha
 ^m_i\frac{u^m_{i+1}-u^m_i}{d_m}-\alpha^m_{i-1}\frac{u^m_i-u^m_{i-1}}{d_m}\bigr]\varphi 
 ^m_id_m= -\sum _{i=-m}^{m-1}\Delta ^{m,-}_i\bigl[\alpha ^m\Delta ^{m,+}(u^m)\bigr]
 \varphi ^m_id_m.  
\end{multline*}
Using test-vectors $\varphi ^m=\frac{1}{d_m}e^m_i$ where
$e^m_i\in\mathbb{R}^m$ is the $i$-th natural basis vector (which
satisfies $\varphi ^m_i\to\delta _x$ if $\frac{i}{m}\to x$ as
$m\to\infty$) yields 
$$
 \dot u^m=-\nabla E^m_\alpha(u^m)=\Delta ^{m,-}\bigl[\alpha ^m\Delta
 ^{m,+}(u^m)\bigr].  
$$
Notice that
\begin{multline*}
 2\frac{d}{dt}\operatorname{avg}(u^m)=\frac{d}{dt}\sum
  _{i=-m}^{m-1}u^m_id_m=\sum 
 _{i=-m}^{m-1}\dot u^m_id_m\\=\sum
 _{i=-m}^{m-1}\Delta _i^{m,-}\bigl(\alpha ^m\Delta ^{m,+}(u^m)d_m=-\sum
 _{i=-m}^{m-1}\alpha^m_i\Delta _i^{m,+}(u) \Delta _i^{m,+}(\mathbf{1})d_m=0
\end{multline*}
for $t\geq 0$. This shows that constant vectors are in the kernel
$\nabla E^m_\alpha$ and thus minimizers of $E^m_\alpha$.

When $m$ is odd, these are the only minimizers of zero energy since
$$
 \alpha ^m_i\geq \min_{j=-m,\dots,m-1}\alpha(x^m_j)\simeq
 (\frac{d_m}{2})^\sigma>0,
$$
and, consequently, $\Delta ^{m,+}(u^m)\equiv 0$ for any minimizer
$u^m$. Thus, for odd $m$, one has that
$$
 u^m(t)\to \frac{1}{2}\sum _{i=-m}^{m-1}u^m_0d_m=
 \operatorname{avg}(u_0^m)\text{ as } t\to\infty,
$$
if $u_0^m$ is the initial vector, just as for $T_1(t)u_0$ at the
continuous level. On the other hand, when $m$ is even,
vectors $H^m(c_1,c_2)$ defined by 
$$
 H^m(c_1,c_2)=\begin{cases} c_1,&-m/2<i\leq m/2,\\ c_2,&i>m/2
 \text{ and }i\leq -m/2\end{cases} 
$$
for any constants $c_1$ and $c_2$ also possess zero energy since
$\alpha ^m_{\pm m/2}=0$. In this case 
\begin{equation}\label{limit}
 u^m(t)\to\operatorname{avg}(u^m_0)+\frac{d_m}{2}[H^m(1,-1)\cdot
 u^m_0]H^m(1,-1) 
\end{equation}
This shows that two distinct solutions can be obtained numerically and that
$E^m_\alpha$ does not have a well-defined unique gradient flow
associated to it as $m\to\infty$ since the evolution clearly depends
on the parity of $m$. Notice also that the large time behavior of
$u^m$ is, for general initial data, incompatible with that of $T_2$
and $T_3$ as well, since
$$
 T_2(t)(v_0+c\,{\bf h})\longrightarrow \frac{1}{2}\int_{-1}^1v_0(x)\,
 dx+c\,{\bf h}\text{ as }t\to\infty,
$$
and
$$
 T_3(t)(v_0+c\,{\bf h})\longrightarrow \bigl( \int_{\Omega_i}v_0(x)\, dx+c
 \bigr) \chi_{\Omega_i}+\bigl( \int_{\Omega_o}v_0(x)\, dx-c
 \bigr) \chi_{\Omega_o}\text{ as }t\to\infty,
$$
if $u_0=v_0+c\,{\bf h}$. Limit \eqref{limit} is in general not the discrete
counterpart of any of these latter limits.\\ 
As it turns out, the behavior of the above discretization, is
compatible with the behavior of solutions of strongly degenerate
equations. To see that, assume that $\alpha$ is a periodic functions,
which is H\"older continuous of a positive exponent, positive
everywhere away from $\pm \frac{1}{2}$ and satisfies
$$
 \alpha (x)\sim|x\pm \frac{1}{2}|^{1+\sigma}\text{ as }x\simeq\mp
 \frac{1}{2},
$$
for some $\sigma>0$. It follows that $\frac{1}{\alpha}\notin
\operatorname{L}^1_\pi$ and that \eqref{lewd} is strongly
degenerate. Define again
$$
 \operatorname{H}^1_{\pi,\alpha}(B)=\big\{ u\in
 \operatorname{L}^2_\pi(B)\, \big |\, \sqrt{\alpha}\, u'\in\operatorname{L}^2_\pi(B)\big\},
$$
and notice, that now, not only ${\bf h}\in \operatorname{H}^1_{\pi,\alpha}(B)$ but also that
$$
 \int _{-1}^1 \alpha\big |\bigl(\varphi_m*{\bf h}\bigr)'(x)\big |^2\, dx\to
 \|\sqrt{\alpha}\, {\bf h} '\|_2^2=0\text{ as }m\to\infty.
$$
This shows that, in the strongly degenerate case, the space
$$
 \big\{ u\in \operatorname{L}^2_\pi(B)\, \big |\, u'\in
 \operatorname{L}^1_\pi(B)\text{ and }\sqrt{\alpha}\,
 u'\in\operatorname{L}^2_\pi(B)\big\}
$$
is not closed and can therefore not be viewed as the ``natural''
domain of the energy functional $E_\alpha$ as in the weakly degenerate
case. Let
$$
 \operatorname{H}^1_{\pi,\alpha,0}(B)=\big\{ u\in
 \operatorname{H}^1_{\pi,\alpha}(B)\, \big |\, \langle
 u,\mathbf{1}\rangle =0,\: \langle u,{\bf h} \rangle =0\big\}
$$
\begin{lem}\label{sdcoercive}
It holds that
$$ 
 \int _{-1}^1 |u(x)|^2\, dx\leq c\, \int _{-1}^1 \alpha(x)|u'(x)|^2\,
 dx\text{ for }u\in \operatorname{H}^1_{\pi,\alpha,0}(B).
$$
\end{lem}
\begin{proof}
Assume that this is not the case. Then a sequence $\seqk{u}$ in
$\operatorname{H}^1_{\pi,\alpha,0}(B)$ can be found such that
$$
 1=\| u_k\| _2^2\geq k\|\sqrt{\alpha}\, u_k'\|_2^2,\: k\in \mathbb{N}.
$$
It follows that $\sqrt{\alpha}u_k' \longrightarrow 0$ in
$\operatorname{L}^2_\pi(B)$. Now let $B_\varepsilon=[|x\pm \frac{1}{2}|\geq
\varepsilon]$ for small $\varepsilon>0$. Then
$$
 \| \chi _{B_\varepsilon}u_k'\|_2^2\leq\, c\| \sqrt{\alpha}\, 
 u'\|_2^2\leq\, \frac{c}{k},\: k\in\mathbb{N}.
$$
It follows that $\chi_{B_\varepsilon}u_k\to 0$ as $k\to\infty$ (along
a subsequence) in $\operatorname{H}^1_{\pi}(B_\varepsilon)$ and therefore
that
$$
 \chi_{B_\varepsilon}u'=0\text{ a.e.}
$$
for any small $\varepsilon>0$. It follows that $u$ must be constant on
$\Omega_i$ and on $\Omega_o$ and, consequently, that
$$
 \operatorname{supp}(u')\subset\{-1/2,1/2\}.
$$
Since $u\in \operatorname{L}^2_\pi(B)$, the distributional derivative
$u'$ has at most order 1 as follows from
$$
 |\langle u',\varphi \rangle |=|\langle u,\varphi '\rangle |\leq \|
 u\| _2\|\varphi '\|_2\leq\, c\| u\| _2\| \varphi '\|_{\infty},\:
 \varphi\in \mathcal{D} _\pi(B).
$$
Combining this with the support condition above, it is concluded that
$$
 u'=A \delta_{-1/2}+B \delta_{1/2}+C \delta'_{-1/2}+D \delta '_{1/2},
$$
for some constants $A,B,C,D$. Since $u\in \operatorname{L}^2_\pi(B)$,
it must then hold that $C=D=0$. One also has that
$$
 A+B=\langle u',\mathbf{1}\rangle = -\langle u,\mathbf{0} \rangle=0,
$$
and consequently that $u=\tilde A+\tilde B {\bf h}$ for some constants
$\tilde A$ and $\tilde B$, which
must both vanish since $u\in
\operatorname{H}^1_{\pi,\alpha,0}(B)$. This clearly yields a
contradiction to $\| u\|_2=1$.
\end{proof}
The form
$$
 a_\alpha(u,v)=\frac{1}{2}\int _{-1}^1 (\sqrt{\alpha}\, u')(x)
 (\sqrt{\alpha}\, v')(x)\, dx
$$
defined on $\operatorname{H}^1_{\pi,\alpha,0}(B)\times
\operatorname{H}^1_{\pi,\alpha,0}(B)$ is therefore coercive and the
associated operator $\mathcal{A}_\alpha$ invertible. The solution $u$
of the corresponding heat equation \eqref{lpwd} with initial datum
$u_0$ therefore satisfies 
$$
 u(t,u_0)\longrightarrow \frac{1}{2}\langle u_0,\mathbf{1}\rangle
 +\langle u_0,{\bf h} \rangle {\bf h}\text{ as }t\to\infty,
$$
just as the numerical solution when $m$ is even. It can be concluded
that, at the discrete level, the distinction between weakly and
strongly degenerate equations can go lost in certain cases.
\begin{rem}
Observe that it is more likely (especially in higher
dimensions) that a numerical scheme will deliver the ``smooth''
solution of the continuous equation rather than the stationary one
(for piecewise constant initial data). This is due to the fact that
latter solution can only be captured if the jumps are on (or close
enough) to the grid and Dirac delta functions at the jump locations
discretize to discrete delta functions (read natural basis
vectors). This is the case in the above example when $m$ is even but
could not hold, e.g.,  for a centered difference scheme based on 
discretizing the first derivative by
$$
 \Delta ^{m,c}_i(u^m)=\frac{u^m_{i+1}-u^m_{i-1}}{2h_n},\:
 i=-m,\dots,m-1, 
$$
i.e. for 
$$
\widetilde E^m_\alpha(u^m)=\frac{1}{2}\sum
_{i=-m}^{m-1}\alpha_i^m[\frac{u^m_{i+1}-u^m_{i-1}}{2h_n}]^2h_n. 
$$
That said, the above example is not pathological. Indeed spectral discretizations in
combination with appropriate discrete quadrature rules for the
discretization of integrals (duality pairings) also capture the
``singular'' rather than the regular solution. This follows again from
the fact that continuous delta functions discretize to discrete delta
functions as is proved in \cite{G062}. 
\end{rem}
\section{Regularization}
Next it is shown that the regularizing interpretation of
\eqref{lpwd} can be view as the limit of the regularized problem
\begin{equation}\label{dlmpme}\begin{cases}
  \dot u=\nabla\cdot\bigl([1/m+\alpha]\nabla u\bigr) &\text{in
  }B\text{ for }t>0,\\u(0)=u_0&\text{in }B.
\end{cases}
\end{equation}
as $m\to\infty$. Start with the regularized energy functional
\begin{equation}\label{denergy}
 E^m _\alpha(u):=\begin{cases}
  \int _B [1/m+\alpha]\,|\nabla u|^2\, dx,& u\in
  \operatorname{H}^1_{\pi}(B),\\
 \infty,&u\in \operatorname{L}^2_\pi(B)\setminus
 \operatorname{H}^1_{\pi}(B).
\end{cases}
\end{equation}
\begin{prop}\label{gamma}
It holds that $E^{m}_\alpha\overset{\Gamma }{\longrightarrow}E_\alpha$
(where $E_\alpha$ is extended by $\infty$ to $\operatorname{L}^2_\pi(B)\setminus
\operatorname{H}^1_{\pi,\alpha}(B)$) as 
$m\to\infty$ with respect to the weak topology of
$\operatorname{H}^1_{\pi,\alpha}(B)$. 
\end{prop}
\begin{rem}
The reason to consider $\Gamma$-convergence is that the domain of
definition of the energy functional changes in the limit. As a
consequence convergence can only be proved in a topology that is too
weak to preserve the equation.
\end{rem}
\begin{proof}
Following e.g. \cite{B14}, $\Gamma$-convergence (with respect with the
weak topology) is defined by the validity of the following estimates
\begin{align*}
\text{(i) }&E_\alpha(u)\leq\liminf_{m\to\infty}E^m_\alpha(u_m)\text{
  for any }\operatorname{H}^1_{\pi}(B)\ni u_m\rightharpoonup  u\text{ in }
 \operatorname{H}^1_{\pi,\alpha}(B)\\\text{(ii) }& \text{For any }
 u\in\operatorname{H}^1_{\pi,\alpha}(B)\text{ there is }\seqm{u}
 \text{ in }\operatorname{H}^1_{\pi}(B),\: u_m \rightharpoonup u , \text{ with }
 E_\alpha(u)=\lim _{n\to\infty}E^m_\alpha(u_m).
\end{align*}
Let $\seqm{u}$ be any sequence in $\operatorname{H}^1_{\pi}(B)$
converging to $u\in \operatorname{H}^1_{\pi,\alpha}$ in the weak topology
of the latter space. Then it clearly holds that
$$
 E_\alpha (u_m)\leq E^m_\alpha (u_m),\: m\in \mathbb{N}, 
$$
and thus
$$
 E_\alpha(u)\leq\liminf _{m\to\infty}E_\alpha(u_m)\leq\liminf_{m\to\infty}
 E^m_\alpha(u_m),
$$
since the first inequality follows from the weak lower semicontinuity
of the norm on the Hilbert space
$\operatorname{H}^1_{\pi,\alpha}(B)$. 
In order to verify the second condition, let $\varphi_m$ be the
mollifier introduced immediately preceding the formulation of Lemma
\ref{alphalemma}. It will be shown that 
$$
 \int _B [\alpha(x)+1/m]|\nabla u^m(x)|^2\, dx\to\int _B \alpha
 |\nabla u|^2\, dx,
$$
as $m\to\infty$ for $u^m:=\varphi_m*u\in
\operatorname{H}^1_{\pi}(B)$. It is a consequence of Lemma
\ref{density} that
$$
 \int _B \alpha(x)|\nabla u^m(x)|^2\, dx\to\int _B \alpha
 |\nabla u|^2\, dx.
$$
To deal with the second term, notice that
$$
 \bigl[\int _B\varphi_m(\cdot-\bar x)\partial_ju(\bar x)\, d\bar
 x\bigr]^2\leq \int _B \frac{\varphi _m(\cdot-\bar x)}{\alpha(\bar
   x)}\,d\bar x\int
 _B \varphi_m(\cdot-\bar x)\alpha(\bar x)\bigl( \partial_ju(\bar x)
 \bigr) ^2\, d\bar x.
$$
As the second factor on the right-hand-side converges to
$\alpha|\partial_ju|^2$ in $\operatorname{L}^1_\pi(B)$ and the first can
be estimated as follows
\begin{multline*}
 \int _B \frac{\varphi _m(\cdot-\bar x)}{\alpha(\bar x)}\leq\, m^2\int
 _{\mathbb{B}(x,1/m)}\frac{d\bar x}{\alpha(\bar x)}=m^2\big\{\int _{\mathbb{B}(x,1/m)\cap
   T_{1/m}(\Gamma)^\mathsf{c}}+\int
 _{\mathbb{B}(x,1/m)\cap T_{1/m}(\Gamma)}\Big\}\frac{d\bar x}{\alpha(\bar x)}\\\leq
 cm^2\int _{\mathbb{B}(x,1/m)}m^\sigma\, d\bar x+cm^2\int
 _{-1/m}^{1/m}\int _{\Gamma \cap \mathbb{B}(x,1/m)}\frac{1}{|\bar
   s|^\sigma}\, d \sigma_\Gamma (\bar y)d\bar s=c\, m^\sigma 
\end{multline*}
it can be concluded that
$$
 \frac{1}{m}\int _B \bigl[\int_B\varphi_m(x-\bar x)\partial_ju(\bar
 x)\, d\bar x\bigr]^2\, dx\leq c\, m^{\sigma-1}\| \varphi_m*\bigl(\alpha
 |\partial_iu|^2\bigr)\|_1 \longrightarrow 0\text{ as }m\to\infty.
$$
The proof is complete. 
\end{proof}
In spite of the fact that both the regularized problem and the
limiting one generate analytic semigroups, solutions of the first do
not converge to solutions of the latter in any strong way. This is due
to the loss of regularity in the limit, where eigenfunctions (and, more
in general, solutions) are no longer smooth (on the degeneration set). In
view of Proposition \ref{gamma}, however, $\Gamma$-convergence proves
a useful tool for the purpose. In fact, known results for gradient
flows showing that, if a sequence of energies $\Gamma$-converges to a
limiting energy, so do the minimizing movements of the corresponding
gradient flows, apply and yield a convergence result. Minimizing
movements $u$ for a, in this context, convex energy functional $E$ on
a Hilbert space $H$ are constructed as (locally uniform) limits
$$
 u(t)=\lim _{h\to 0+}u^h(t),
$$
of approximating piecewise constant functions $u^h(t)=u^{h,\lfloor t/h\rfloor}$
obtained by recursive minimization 
\begin{equation}\label{minp}
 u^{h,k+1}={\arg\min} _{v\in H}\big\{ E(v)+\frac{1}{2h}\|
 v-u^{h,k}\|^2 _H\big\} 
\end{equation}
starting from an initial datum $u^h_0$. Latter essentially amounts to
solutions of the Euler scheme with time step $h>0$ for the
corresponding gradient flow.
\begin{rem}\label{minsemi}
Observe that, when $E$ is a quadratic and therefore differentiable
functional, and the linear operator $\mathcal{A}=\nabla E$ is the generator of a
strongly continuous analytic semigroup of contractions as is the
case for $E^m_\alpha$ and $E_\alpha$, then the minimization problem
\eqref{minp} is equivalent to
$$
 (1+h \mathcal{A} )v=u^{h,k}\text{ in }\operatorname{H}^{-1}_{\pi}(B)\text{
   or }\operatorname{H}^{-1}_{\pi,\alpha}(B),
$$
for $\mathcal{A}=\mathcal{A}^m_\alpha$ or
$\mathcal{A}=\mathcal{A}_\alpha$ respectively. Consequently, one has
that 
$$
 u^{h,k}=(1+h \mathcal{A} )^{-1}u^{h,k-1}=(1-h \mathcal{A} )^{-k}u^h_0.
$$
When $u^h_0=u_0\in \operatorname{L}^2_\pi(B)$ and $kh\to t$, semigroup
theory (see \cite{Gol85}) implies that
$$
 (1+h \mathcal{A} )^{-k}u_0\to e^{-t \mathcal{A} }u_0=T_1(t)u_0\text{ as }h\to 0.
$$
In this case, the minimizing movement originating in $u_0$ coincides
with the solution that was previously constructed by the semigroup
approach.
\end{rem}
The following theorem is stated and proved in \cite[Chapter 11]{B14}.
\begin{thm}
Let $\seqm{F}$ be a sequence of equi-coercive, lower semicontinuous,
positive convex energies $\Gamma$-converging to $F$, and let $x_0^m\to
x_0$ with $\sup _{m\in \mathbb{N}}F^m(x^m_0)<\infty$. Then the
sequence of minimizing movement $u_m$ for $F^m$ starting in $x^m_0$
converges to the minimizing movement $u$ for $F$ originating in $x_0$.
\end{thm}
This theorem yields the following result in the situation considered
in this paper.
\begin{thm}\label{grad2grad}
Let $\operatorname{H}^1_\pi(B) \ni u_0^m\to u_0$ in 
$\operatorname{H}^1_{\pi,\alpha}(B)$ as $m\to\infty$ be such that 
$$
 \sup _{m\in \mathbb{N}}\| u^m_0\|
 _{\operatorname{H}^1_{\pi,\alpha}(B)}\leq\, c<\infty.
$$
Then the solution $u^m(\cdot,u_0^m):[0,\infty)\to
\operatorname{H}^1_\pi(B)$ of \eqref{dlmpme} with initial datum
$u^m_0$ converges to the solution of limiting equation \eqref{lmpme}
with initial datum $u_0$.
\end{thm}
\begin{proof}
It follows from Remark \ref{minsemi}  that the minimizing movements
for $E^m_\ga$ and $E_\alpha$ coincide with the solutions $T^m_\alpha(t)u^m_0$ and
$T_\alpha(t)u_0$ given by the analytic semigroups $T^m_\alpha$ and 
$T_\alpha$ generated by the operators
$A^m_\alpha=\nabla\cdot\bigl([\frac{1}{m}+\alpha]\nabla\cdot\bigr)$
and $A_\alpha$ on $\operatorname{L}^2_\pi(B)$, respectively.

Now equi-coercivity follows from
$$
 E^m_\alpha (u)\geq E_\alpha(u),\: u\in \operatorname{L}^2_\pi(B),
$$
and the coercivity on $E_\ga$ on
$\operatorname{H}^1_{\pi,\alpha}(B)$. As for weak lower
semicontinuity of $E^m_\alpha $, take a sequence $\operatorname{H}^1_\pi(B)\ni
u_k \rightharpoonup u$ in $\operatorname{H}^1_{\pi,\alpha}(B)$. It is
easily verified that, for any fixed $\varepsilon>0$,
$$
 u_k \rightharpoonup u\text{ in }\operatorname{H}^1_\pi
 \bigl(T_\varepsilon(\Gamma)^\mathsf{c} \bigr) \text{ as }k\to\infty,
$$
where, as before, $T_\varepsilon(\Gamma)$ is the tubular neighborhood of
$\Gamma$ of ``thickness'' $\varepsilon>0$. It follows that, for any fixed $\varepsilon>0$,
\begin{align*}
 \int _{T_\varepsilon(\Gamma)^\mathsf{c}}[\frac{1}{m}+\alpha]|\nabla
 u|^2\, dx&\leq\liminf_{k\to\infty}\int
 _{T_\varepsilon(\Gamma)^\mathsf{c}}[\frac{1}{m}+\alpha]|\nabla
 u_k|^2\, dx\\ &\leq\liminf_{k\to\infty}\int
 _{B}[\frac{1}{m}+\alpha]|\nabla u_k|^2\, dx.
\end{align*}
Thus, if $u\in \operatorname{H}^1_\pi(B)$, then
$$
 \int _B [\frac{1}{m}+\alpha]|\nabla u|^2\, dx=\lim_{\varepsilon\to 0+}\int
 _{T_\varepsilon(\Gamma)^\mathsf{c}}[\frac{1}{m}+\alpha]|\nabla 
 u|^2\, dx\leq\liminf_{k\to\infty}\int
 _{B}[\frac{1}{m}+\alpha]|\nabla u_k|^2\, dx,
$$
whereas, if $u\in \operatorname{H}^1_{\pi,\alpha}(B)\setminus
\operatorname{H}^1_\pi(B)$, one has that
$$
 \liminf_{k\to\infty}\int_{B}[\frac{1}{m}+\alpha]|\nabla u_k|^2\, dx
 \geq \int _{T_\varepsilon(\Gamma)^\mathsf{c}}[\frac{1}{m}+\alpha]
 |\nabla u|^2\, dx\to\infty\text{ as }\varepsilon\to 0
$$
\end{proof}
\begin{rem}
Notice that the existence of approximating sequences for intial data
such as those needed for Theorem \ref{grad2grad} follows from
the construction of recovery sequences performed in the proof of
Proposition \ref{gamma}.
\end{rem}
%\section{Eigenvalues and Eigenfunctions}
\section{Appendix}
It remains to prove that \eqref{kernel} and \eqref{kernelasym} are
valid. It is well-known that 
$$
 \mathcal{F}^{-1}\bigl( \frac{1}{|\xi|^\varepsilon}\bigr)
 =\frac{c_\varepsilon}{|x|^{n-\epsilon}} 
$$
on $\mathbb{R}^n$ for $n=1,2$. Using this and classical arguments
based on the Poisson summation formula it can be inferred that, for
the discrete Fourier transform of periodic functions
$$
 \mathcal{F}^{-1}\bigl( \frac{1}{|k|^\varepsilon}\bigr) =
 \frac{c_\varepsilon}{|x|^{n-\epsilon}}+h_\varepsilon(x),
$$
for a $\operatorname{C}^\infty$-function $h_\varepsilon$. Indeed we
have the following kernel characterizations.
\begin{lem}\label{kernelproof}
Let $\epsilon\in(0,1)$ and assume that the fractional derivative
be given by 
$$
 |\nabla|^{-\varepsilon}|\nabla u|=\mathcal{F}^{-1}
 \operatorname{diag}\big\{\frac{1}{|k|^\epsilon}\big\}\mathcal{F}
 \bigl( |\nabla u|\bigr)=N_\varepsilon(|\nabla u|),
$$
where $\nabla$ is taken to be $\partial$ when $n=1$.
Then, for $n=1,2$,
\begin{equation*}
 |\nabla|^{-\varepsilon}|\nabla u|= \int _B G^n_\varepsilon(x-\tilde
 x)|\nabla u|(\tilde x)\, d\tilde x, 
\end{equation*}
for a periodic function $G^n_\epsilon$ satisfying
$$
 G^n_\epsilon (x)=c_\epsilon \frac{1}{|x|^{2-\epsilon}}+h^n_\epsilon(x),\:
 x\in B^n, 
$$
and a function $h^n_\epsilon\in \operatorname{C}^\infty$.
\end{lem}
\begin{proof}
By definition, one has that
$$
 \widehat G^n_\epsilon (k)=\frac{1}{|k|^\epsilon},\: k\in
 \mathbb{Z}_*^2:=\mathbb{Z}^n\setminus\{ 0\}.
$$
This means that
$$
 G^n_\epsilon(x)=\sum
 _{k\in\mathbb{Z}^n_*}\frac{1}{|k|^\epsilon}e^{\pi ik\cdot x}=\sum
 _{k\in\mathbb{Z}^n_*}\frac{\eta(k)}{|k|^\epsilon}e^{\pi ik\cdot x},
$$
where $\eta\in \operatorname{C} ^\infty(\mathbb{R}^n)$ is a cut-off
function with 
$$
 \eta(x)=\begin{cases} 0,&|x|\leq 1/4,\\ 1,&|x|\geq 1/2.\end{cases}
$$
Notice that Poisson summation formula yields
$$
 G_\epsilon^n(x)=\sum_{k\in\mathbb{Z}^n_*}\frac{\eta(k)}{|k|^\epsilon}
 e^{\pi ik\cdot x}=g^n_\epsilon(x)+\sum_{k\in
   \mathbb{Z}^n_*}g^n_\epsilon(x+k),\: x\in B^n,
$$
where $g^n_\epsilon=\mathcal{F}\bigl(\eta |\cdot |^{-\epsilon}\bigr)$ is rapidly
decreasing (faster than the reciprocal of any polynomial) as the
Fourier transform of a smooth function, and satisfies 
$$
 g^n_\epsilon=c_\epsilon |\cdot|^{\epsilon -1}+\mathcal{F}\bigl([\eta-1] |\cdot
 |^{-\epsilon}\bigr),\: x\in \mathbb{R}, 
$$
where the second addend is a smooth function as the Fourier transform
of a compactly supported function. Combining everything together
yields the claimed decomposition with 
$$
 h^n_\epsilon=\mathcal{F}\bigl([\eta-1] |\cdot|^{-\epsilon}\bigr)+\sum _{k\in
   \mathbb{Z}^*}g^n_\epsilon(\cdot+k).
$$
\end{proof}
The following lemma gives a proof of \eqref{kernelasym}.
\begin{lem}
If $n=1$, set $u_0=\chi _{[-1/2,1/2]}$ (or the characteristic function
of any interval) and, if $n=2$, let $u_0=\chi _\Omega$ for a domain
$\Omega\subset B$ with smooth boundary $\Gamma$ (or a finite 
combination of such characteristic functions of non-intersecting
domains).
Then, for $n=1,2$, with the same interpretations as in the previous
lemma, one has that
$$
 |\nabla|^{-\varepsilon}|\nabla u_0|(x)\sim
 d(x,\Gamma)^{\varepsilon-1}\text{ for }d(x,\Gamma)\sim 0.
$$
\end{lem}
\begin{proof}
Using the kernel representation given in Lemma \ref{kernelproof} and
the fact that $\partial u_0=\delta _{-1/2}-\delta _{1/2}$ yields that
$$
 \bigl( G^1_\epsilon *_\pi |\partial u_0|\bigr) (x)=\, c_\epsilon\bigl[
 \frac{1}{|x+1/2|^{1-\epsilon}}+\frac{1}{|x-1/2|^{1-\epsilon}}\bigr]
 +\text{ smooth term},\: x\in(-1,1),
$$
and the claims follow. When $n=2$, it is easily seen that $\nabla
\chi_\Omega =\nu _\Gamma \delta _\Gamma$ for
$$
 \langle \nu _\Gamma \delta _\Gamma,\varphi \rangle = \int _\Gamma \nu
 _\Gamma (x)\cdot \varphi(x)\, d \sigma _\Gamma(x),
$$
and where $\nu _\Gamma \delta _\Gamma$ can be interpreted as a vector
measure. Then its total variation measure $|\nu _\Gamma \delta
_\Gamma|$ is simply given by $\delta_\Gamma$. It follows that
$$
 N_\epsilon \bigl(|\nabla u_0|\bigr) =\int _B G^2_\varepsilon(x-\tilde
 x)|\nabla u_0|(\tilde x)\, d\tilde x=c_\varepsilon \int _\Gamma
 \frac{1}{|x-\tilde y|^{2-\varepsilon}}\, d \sigma _\Gamma(\tilde
 y)+\text{smooth term} 
$$
Next fix a point $x$ in the vicinity of $\Gamma$. Denote by $y_x$ the
point on $\Gamma$ closest to $x$. Exploiting the fact that the
curve $\Gamma$ is smooth and compact and has hence bounded curvature, it is
seen that
$$
 |x-\tilde y|^2=\bigl( |x-y_x|\pm |\tilde y_x-\tilde y|\bigr)
 ^2+|y_x-\tilde y|^2\sim (r\pm c s^2)^2+s^2 \sim r^2+s^2,
$$
for $\tilde y$ in a small fixed ball $\mathbb{B}_\Gamma(y_x,\delta)$
uniformly in $x\in T_\delta(\Gamma)$ for a (without loss of
generality) common $\delta>0$. Here $r=d(x,\Gamma)$ and $s=|y_x-\tilde
y_x|$ where $\tilde y_x$ is the orthogonal projection of $\tilde y$ to
the line spanned by $\tau(y_x)$ in
the local coordinate system given by $\tau(y_x)$ and $\nu(y_x)$, the unit
tangent and outward normal to $\Gamma$ at $y_x$, respectively. See figure
below. It follows that
$$
 N_\epsilon \bigl(|\nabla u_0|\bigr)(x)\sim\int _{-\delta}^\delta
 (s^2+r^2)^{\varepsilon/2-1}\, ds\sim r^{\varepsilon-1}\int
 _{-\infty}^\infty(1+\sigma^2)^{\varepsilon/2-1}\, d \sigma=c\,
 d(x,\Gamma)^{\varepsilon-1},
$$
which yields the claim since $\varepsilon<1$.
\begin{center}\begin{tikzpicture}
 \tkzInit[ymin=-2,ymax=3,xmin=-6,xmax=6];
 \draw [domain=-3.5:3.5] plot ({\x}, {-0.05*\x*\x});
 \draw [gray,domain=-3.5:3.5] plot ({\x}, {-0.05*\x*\x-1});
 \draw [gray,domain=-3.5:3.5] plot ({\x}, {-0.05*\x*\x+1});
 \node at (-3.7,-0.6) {$\Gamma$};
 \node[thick] at (0,0) {\tiny$\bullet$};
 \node at (-0.2,-0.2) {$y_x$};
 \draw[dashed,->] (0,-1.5)--(0,1.5);
 \node[right] at (0,1.5) {$r$};
 \node[thick] at (0,0.6) {\tiny$\bullet$};
 \node[left] at (0,0.6) {$x$};
 \draw[dashed,->] (-3,0)--(3,0);
 \node[right] at (3,0) {$s$};
 \node[thick] at (2.25,{-0.05*2.25^2}) {\tiny$\bullet$};
 \node[below] at (2.25,{-0.05*2.25^2}) {$\tilde y$};
 \node at (2.25,0) {\tiny$\bullet$};
 \node[above] at (2.25,0) {$\tilde y_x$};
\end{tikzpicture}\end{center}
\end{proof}
%\bibliography{../../lite}

\end{document}